\documentclass[12pt]{amsart}

\usepackage{cancel}
\usepackage{latexsym, amssymb, amsmath,esint}
\usepackage{soul}
\usepackage{amsfonts, graphicx}
\usepackage{graphicx,color}
\newcommand{\be}{\begin{equation}}
\newcommand{\ee}{\end{equation}}
\newcommand{\beq}{\begin{eqnarray}}
\newcommand{\eeq}{\end{eqnarray}}

\newtheorem{thm}{Theorem}[section]
\newtheorem{conj}{Conjecture}[section]

\newtheorem{lma}{Lemma}[section]
\newtheorem{prop}{Proposition}[section]
\newtheorem{cor}{Corollary}[section]
\newtheorem{defn}{Definition}[section]
\theoremstyle{remark}
\newtheorem{rem}{Remark}[section]
\numberwithin{equation}{section}

\def\be{\begin{equation}}
\def\ee{\end{equation}}
\def\bee{\begin{equation*}}
\def\eee{\end{equation*}}

\def\lf{\left}
\def\ri{\right}

\def\Ric{\text{\rm Ric}}
\def\Rm{\text{\rm Rm}}

\def\cR{ \mathcal{R}}

\def\wh{\widehat}
\def\wt{\widetilde}
\def\la{\langle}
\def\ra{\rangle}

\def\p{\partial}

\def\heat{\lf(\frac{\p}{\p t}-\Delta_t\ri)}

\def\wn{\wt\nabla}

\def\e{\varepsilon}

\def\a{{\alpha}}
\def\b{{\beta}}

\begin{document}

\title[]
{Continuous metrics and a conjecture of Schoen}

 \author{Man-Chun Lee}
\address[Man-Chun Lee]{Department of Mathematics, The Chinese University of Hong Kong, Shatin, Hong Kong, China}
\email{mclee@math.cuhk.edu.hk}

\author{Luen-Fai Tam}
\address[Luen-Fai Tam]{The Institute of Mathematical Sciences and Department of Mathematics, The Chinese University of Hong Kong, Shatin, Hong Kong, China.}
 \email{lftam@math.cuhk.edu.hk}


\renewcommand{\subjclassname}{
  \textup{2010} Mathematics Subject Classification}

\date{\today}

\begin{abstract} A classical theorem in conformal geometry states that on a manifold with non-positive Yamabe invariant, a smooth metric achieving the invariant must be Einstein. In this work, we extend it to the singular case and show that in all dimension, if a continuous metric is smooth outside a compact set of high co-dimension and achieves the Yamabe invariant, then the metric is Einstein away from the singularity and can be extended to be smooth on the manifold in a suitable sense. As an application of the method, we prove a Positive Mass Theorem for asymptotically flat manifolds with analogous singularities.
\end{abstract}


\maketitle

\markboth{Man-Chun Lee, Luen-Fai Tam}{Continuous metrics and a conjecture of Schoen}
\section{introduction}

In this work, we want to study the following conjecture of Schoen:

\begin{conj}[Conjecture 1.5 in \cite{LiMantoulidis2019}] \label{conj-1}
Let $M^n$ be a compact manifold with $\sigma(M)\leq 0$. Suppose $g$ is an $L^\infty$ metric on $M$ such that $g$ is smooth away from a closed, embedded submanifold $\Sigma$ with co-dimension $\geq 3$ and satisfies $\mathcal{R}(g)\geq 0$ outside $\Sigma$, then $\Ric(g)=0$ and $g$ can be extended smoothly on $M$.
\end{conj}
Here $\sigma(M)$ is the $\sigma$-invariant or Yamabe invariant of a compact smooth manifold $M$ introduced by Schoen \cite{Schoen1987}, see also the work of Kobayashi \cite{Kobayashi1987}. Moreover, $g$ is said to be $L^\infty$ metric if $g$ is a measurable section of $\mathrm{Sym}_2(T^*M)$ such that $\Lambda^{-1}h\leq g\leq \Lambda h$ almost everywhere on $M$ for some $\Lambda>1$ and smooth metric $h$.
Let us first recall its definition.  For a conformal class  $\mathcal{C}$ of smooth Riemannian metrics $g$, the {\it Yamabe constant of $\mathcal{C}$} is defined as:
$$
Y(\mathcal{C})=\inf_{g\in\mathcal{C}}\displaystyle{\frac{\int_M\mathcal{R}_g\,d\mu_g}{({\mathrm{Vol}}(M,g))^{1-\frac2n}}}.
$$
where $\mathcal{R}_g$ is the scalar curvature and $\mathrm{Vol}(M,g)$ is the volume of $M$ with respect to $g$. The {\it Yamabe invariant}  is defined as
$$
\sigma(M)=\sup_{\mathcal{C}}Y(\mathcal{C}).
$$
The supremum is taken among all conformal classes of smooth metrics. It is finite, see \cite{Aubin1976}. Since it is well-known that if $\sigma(M)\le0$ then a smooth metric with unit volume and with scalar curvature bounded  below by $\sigma(M)$ is Einstein, Conjecture \ref{conj-1} can be extended to the following:

\begin{conj}\label{conj-2}
Let $M^n$ be a compact manifold with $\sigma(M)=\sigma_0\leq 0$. Suppose $g$ is an $L^\infty$ metric on $M$ with unit volume such that $g$ is smooth away from a closed, embedded submanifold $\Sigma$ with co-dimension $\geq 3$ and satisfies $\mathcal{R}(g)\geq \sigma_0$ outside $\Sigma$, then $g$ is Einstein  and $g$ can be extended smoothly on $M$.
\end{conj}

The conjectures are motivated by another conjecture by Geroch 
 that a torus cannot admit a metric with positive scalar curvature, and  metrics with non-negative scalar curvature must be flat, see \cite{Geroch1975,KazdanWarner1975}. The conjecture was proved by Schoen-Yau \cite{SchoenYau1979-3,SchoenYau1979-4} for $n\leq 7$ using minimal surface method and Gromov-Lawson \cite{GromovLawson1980} for general $n$ using Atiyah-Singer index theorem for a twisted spinor bundle on a spin manifold. On the other hand, metrics with low-regularity arise naturally in the compactness theory and in the study of Brown-York quasi-local mass \cite{ShiTam2002}. It is therefore natural to understand metrics with low-regularity and with scalar curvature bounded from below. Unlike the co-dimension three singularity, in case of co-dimension one and co-dimension two singularities, without some assumptions in addition to $L^\infty$ on the metric, one cannot expect that the metric is Ricci flat outside the singular sets even if the metric has nonnegative scalar curvature in the smooth part.   We refer interested readers to the discussions in \cite{LiMantoulidis2019}. 

When $n=3$, Conjecture \ref{conj-1} was confirmed by Li-Mantoulidis using minimal surface method. See also the related results in \cite{ChengLeeTam2021} on Conjecture \ref{conj-2}.
Our main result is the following:

\begin{thm}\label{main-Thm-Ric}
Let $M^n$ be a compact manifold with $\sigma_0=\sigma(M)\leq 0,n\geq 3$ where $\sigma(M)$ is the $\sigma$-invariant of $M$. Suppose $g_0$ is a continuous metric on $M$ such that $g_0\in C^\infty_{loc}(M\setminus \Sigma)$ for some compact set $\Sigma$ of co-dimension at least $2+a$ for some $a>0$, $\mathrm{Vol}(M,g_0)=1$ and $\cR(g_0)\geq \sigma_0$ on $M\setminus \Sigma$. Then there is a homeomorphism $\Psi: M\to M$ which is bi-Lipschitz with respect to some smooth background metric and a Einstein metric $G$ on $M$ with unit volume and with scalar curvature $\sigma_0$ so that
\begin{enumerate}
  \item [(i)] $\Psi$ smooth on $M\setminus\Sigma$. Moreover $g_0=\Psi^*G$ in $M\setminus \Sigma$. In particular, $g_0$ is Einstein on $M\setminus\Sigma$ with scalar curvature $\sigma_0$.
  \item [(ii)]  $\Psi:(M,d_{g_0})\to (M,d_{G})$ is an isometry as metric spaces, where $d_{g_0}$ and $d_{\wt g}$ are the distance metrics induced by $g_0$ and $G$ respectively.
\end{enumerate}

Furthermore, if $\Sigma$ consists of only isolated points,  then $g_0$ is a smooth metric with respect to   a possibly different smooth structure on $M$.
\end{thm}
For a complete smooth Riemannian manifold $(M^n, h)$, a compact set  $\Sigma$ of $M$ is said to have co-dimension at least $\mathfrak{l}_0>0$ if there exist $b>0$ and $C>0$  such that for all $0<\e\le b$
\be\label{e-codim}
V_{h}(\Sigma(\e))\le   C\e^{\mathfrak{l}_0}
\ee
where $V_h$ is the volume with respect to $h$ and
$$
\Sigma(\e)=\{x\in M|\ d_h(x,\Sigma)<\e\}.
$$
2It is easy to see that the definition does not depend on the smooth metric $h$. Moreover, if the upper Minkowski dimension of $\Sigma$ is less than $\mathfrak{l}_0$, then the co-dimension of $\Sigma$ is at least  $  n-\mathfrak{l}_0>0$. In case $\Sigma$ is an embedded submanifold of dimension $k$, then its co-dimension is at most $n-k$. It is not difficult to construct example of $\Sigma$ with non-integral co-dimension. For instance, one might consider $\Sigma=\{p_k\}_{k=1}^\infty\subset M^n$ with $d_h(p_k,p_{k+1})\leq  k^{-\a}$ for some $\a>1$ so that $p_k\to p_\infty\in M$. In this way, the upper Minkowski dimension of $\Sigma$ will be at most  $  n\a^{-1}\in (0,1)$ and hence the co-dimension is at least $  n-n\a^{-1}$.  Hence Theorem \ref{main-Thm-Ric} can be applied to singularities of this kind of $\Sigma$ with $\a>\frac{n}{n-1}$. See Corollary \ref{c-points} for more details. In particular, Theorem \ref{main-Thm-Ric} partially confirms Schoen conjecture in the category of $C^0$ metrics. As an application of the method, we prove that $C^0$ metrics with singularity in form of Theorem~\ref{main-Thm-Ric} has global scalar curvature lower bound in a weak sense, see Corollary~\ref{Cor:scalar-lower-C0}.

 On the non-compact side, Schoen and Yau \cite{SchoenYau1979,SchoenYau1979-2,SchoenYau2017} proved the positive mass theorem which asserts that the Arnowitt-Deser-Misner (ADM) mass of each end of an $n$-dimensional asymptotically flat (AF) manifold with nonnegative scalar curvature is non-negative and if the ADM mass of an end is zero, then the manifold is isometric to the Euclidean space, see also \cite{Bartnik1986,ParkerTaubes1982,Witten1981} for the earlier works. The method of proof of Theorem \ref{main-Thm-Ric}   also enables us to prove the following positive mass theorem:
\begin{thm}\label{main-Thm-PMT}
Let $(M^n,g_0)$ be a AF manifold with $n\geq 3$, $g_0$ is a continuous metric on $M$ such that $g_0$ is smooth away from some compact set $\Sigma$ of $M$ of co-dimension at least $\geq 2+a$ for some $a>0$. Suppose $\cR(g_0)\geq 0$ outside $\Sigma$, then the ADM mass of each end is nonnegative. Moreover, if the ADM mass of one of the ends is zero, then $(M,g_0)$ is isometric to $(\mathbb{R}^n,g_{euc})$ as a metric space and is flat outside $\Sigma$.
\end{thm}

  When the singular set is of lower co-dimension, the related positive mass theorem has been studied by various authors, see \cite{Miao2003,JiangShengZhang2020,Lee2013,LeeLeFloch2015,LiMantoulidis2019,ShiTam2016} and the reference therein. 
Unlike  most of the previous results, we do not assume any $L^p$ bounds on the  first derivative of the metric. We only assume that the metric is $C^0$ at the singular set.

The paper is organized as follows. In Section~\ref{Sec: RF}, we will collect some useful result on the existence of the Ricci-Deturck flows. In Section~\ref{ss-MP}, we will prove a local maximum principle and monotonicity formula along the Ricci-Deturck flows.  In Section~\ref{Sec:compact}, we will prove Theorem~\ref{main-Thm-Ric}. In Section~\ref{Sec:AF}, we will consider the asymptotic flat manifolds and prove Theorem~\ref{main-Thm-PMT}. In this work, the dimension of any manifold is assumed to be at least three.

\section{Preliminaries}\label{Sec: RF}

We would like to regularize the metric using the Ricci flow. We will start with the Ricci-Deturck flow with background metric $h$. We will follow \cite{Simon2002} to call it $h$-flow to emphasis the dependence. We first need some basic facts about the flow. In the following, complete manifolds are referring to either complete non-compact manifold or compact manifolds without boundary.

\subsection{Basic facts on $h$-flow}

Let $(M,h)$ be a complete Riemannian manifold such that for all $i\in\mathbb{N}$, there is $k_i>0$ so that
\begin{equation}\label{background-h-regular}
|\wt\nabla^i \Rm(h)|\leq k_i
\end{equation}
where $\wt \nabla$ denotes the covariant derivative with respect to $h$. By the work of Shi \cite{Shi1989}, we may perturb metrics with bounded curvature slightly so that \eqref{background-h-regular} holds.

A smooth family of metrics $ g(t)$ on $M\times (0,T]$ is said to be a solution to the $h$-flow if it satisfies
\begin{equation}\label{equ:h-flow}
\left\{
\begin{array}{ll}
\partial_t g_{ij}=-2R_{ij}+\nabla_i W_j +\nabla_j W_i;\\
W^k= g^{pq}\left( \Gamma^k_{pq}-\wt\Gamma^k_{pq} \right).
\end{array}
\right.
\end{equation}

To regularize a non-smooth metric, it is also common to consider the Ricci flow which is a smooth family of metric $ \wh g(t)$ satisfying
\begin{equation}
\frac{\partial}{\partial t}  \wh  g_{ij}=-2\Ric(  \wh g)_{ij}.
\end{equation}
If the initial metric $g_0$ is smooth, it is well-known that the Ricci flow is equivalent to the Ricci-Deturck flow in the following sense. Let

$\Phi_t$ be the diffeomorphism given by
\begin{equation}\label{e-Phi}
\left\{
\begin{split}
\frac{\partial}{\partial t}\Phi_t(x)&=-W\left(\Phi_t(x),t \right);\\
\Phi_0(x)&=x.
\end{split}
\right.
\end{equation}
Then the pull-back of the Ricci-Deturck flow $ \wh g(t)=\Phi_t^*  g(t)$ is a Ricci flow solution with $\wh g(0)= g(0)=g_0$. We will interchange between the Ricci flow and Ricci-Deturck flow depending on the purpose.

Before we state the ingredients, we fix some notations. For $\sigma>1$, a continuous metric $g$ is said to be $\sigma$-close to $h$ if
\begin{equation}
\sigma^{-1}h \leq g\leq \sigma h.
\end{equation}
We will also use $a\wedge b$ to denote $\min\{a,b\}$ for any $a,b\in \mathbb{R}$.

In \cite{Simon2002}, Simon obtained the following regularization result for continuous metrics using the $h$-flow (i.e. Ricci-Deturck flow), see also \cite{Burkhardt2019,KochLamm2012,Shi1989}.
\begin{thm}[Simon, Theorem 5.2 in \cite{Simon2002}]\label{Simon-Theorem} There is $\e_n>0$ such that the following is true: Let $(M,h)$ be a complete manifold satisfying \eqref{background-h-regular}. If $g_0$ is a continuous metric on $M$ such that $g_0$ is $1+\e_n$ close to $h$, then the \eqref{equ:h-flow} admits a smooth solution $g(t)$ on $M\times (0,T_0]$ for some $T_0(n,k_0)>0$ so that
\begin{enumerate}
\item[(i)] $$\lim_{t\to 0}\sup_{\Omega}|g(t)-g_0|=0, \; \forall \; \Omega\Subset M;$$
\item[(ii)] For all $i\in \mathbb{N}$, there is $C_i>0$ depending only on $n,k_0,...k_i$ so that
$$\sup_M |\wt \nabla^i g(t)|\leq \frac{C_i}{t^{i/2}}.$$
\item[(iii)]$g(t)$ is $1+2\e_n$ close to $h$ for all $t\in (0,T_0]$.
\end{enumerate}
Here the norm $|\cdot|$ and connection $\wt\nabla$ are with respect to $h$.
\end{thm}

\begin{rem}
When the initial metric is only $C^0$, the properties of the regularizing Ricci flow has been extensively studied by Burkhardt-Guim in \cite{Burkhardt2019}. Since we need to perform some local analysis away from singular set, we stick with the original approach by Simon.
\end{rem}

In particular, it was shown that the continuous metric $g_0$ can be smoothed such that the curvature of $g(t)$ is bounded above by $\a t^{-1}$ for some $\a>0$. In \cite{HuangTam2018}, Huang and the second named author  improve the estimate such that $\a$ can be made arbitrarily close to $0$ if $g(t)$ is further close to $h$ in $C^0$ topology.
\begin{prop}\label{hFlow-improve}
For any $\delta>0, k_0>0$, there is $T_1(n,\delta, k_0),\sigma(n,\delta)>0$ such that the following holds. Let $(M,h)$ be a complete manifold with $|\Rm(h)|+|\wn\Rm(h)|+|\wn^2\Rm(h)|\le k_0$. If $g(t)$ is a smooth solution to the $h$-flow on $M\times [0,S]$ obtained in Theorem~\ref{Simon-Theorem} and $g_0$ is $1+\sigma(n,\delta)$ close to $h$, then we have
$$|\wt \nabla g(t)|^2+|\wt\nabla^2 g(t)|+|\Rm_{g(t)}|\leq \frac{\delta }{t}$$
on $M\times (0,T_1\wedge S]$ where $\wt \nabla$ is the covariant derivative with respect to $h$.
\end{prop}
\begin{proof}
It follows from \cite[Lemma 5.1, Lemma 5.2]{HuangTam2018}. The proof in the complete non-compact case can easily be adapted to the compact case by removing the cutoff function in the maximum principle argument.
\end{proof}
We should remark that   $\sigma$ does not depend on $k_0$, even though the time interval may shrink if $k_0$ is large.

The next Proposition illustrates that the $h$-flow is locally uniformly regular up to $t=0$ if the initial metric is locally regular.

\begin{prop}\label{prop-local-spacetime}
Under the assumption of Theorem~\ref{Simon-Theorem}, if $g_0$ is smooth on $\Omega\Subset M$ so that $\sup_\Omega \sum_{m=1}^i|\wt \nabla^m g_0|\leq L_i$, then for all $\Omega'\Subset \Omega$, we have
$$\sup_{\Omega'\times [0,T]}|\wt \nabla^i g(t)|\leq C_0$$
for some $C_0>0$ depending only on $n,i,k_0,...,k_{i},L_1,...,L_i,\Omega'$ and $\Omega$.
\end{prop}
\begin{proof}
The proof is identical to that of \cite[Lemma 4.2]{Shi1989} except the background metric is chosen to be $h$ instead of the initial metric $g_0$. See also \cite{Simon2002}.
\end{proof}

\bigskip
\subsection{Regularizing $C^0$ metrics on compact manifolds}\label{subsection-compact-regularize}

Our main goal of this subsection is to prove the following.
\begin{prop}\label{l-regularization} Let $(M^n,g_0)$ be a compact Riemannian manifold with a $C^0$ metric $g_0$ which is smooth outside some compact subset $\Sigma$.  Then for any $\delta>0$,  there is a smooth metric $h$ such that the $h$-flow \eqref{equ:h-flow} has a solution $g(t)$ on $M\times(0,T]$ for some $T>0$ with the following properties:
\begin{enumerate}
  \item[(i)] $g(t)\to g_0$ in $C^0(M)$ and $g(t)\to g_0$ in $C^\infty_{loc}(M\setminus \Sigma)$ as $t\to 0$.
  \item[(ii)] $\frac 12 h\le g(t)\le 2 h$ in $M\times[0,T]$
  \item[(iii)]
  $$|\wn g(t)|^2_h+|\wn^2 g(t)|_h +|\Rm(g(t))|_{g(t)}\le \frac \delta t
  $$
  on $M\times(0,T]$. Here $\wn$ is the covariant derivative with respect to $h$.
\end{enumerate}

\end{prop}

Let $(M^n,g_0)$ be a  complete Riemannian manifold without boundary and let $\Sigma$ be a compact set of $M^n$. Assume $g_0$ is in $C^0(M)$ and $g\in C^\infty_{loc}(M\setminus \Sigma)$. For any $a>0$, denote
\be\label{e-sigma-a}
\Sigma(a):=\{x\in M|\ \ d_{g_0}(x,\Sigma)<a\}
\ee
where $d_{g_0}$ is the distance function induced by $g_0$.

We start with an approximation of $g_0$.
\begin{lma}\label{smooth-approximation}
For a continuous metric $g_0$, there is a sequence of smooth metrics $g_{i,0}$ on $M$ such that $g_{i,0}=g_0$ outside $\Sigma(i^{-1})$ and $g_{i,0}$ converges to $g_0$ in $C^0$ topology.
\end{lma}
\begin{proof}
The proof is identical to that \cite[Lemma 4.1]{ShiTam2016} except we don't have the additional uniform $W^{1,p}$ structure.
\end{proof}
\bigskip

\begin{proof}[Proof of Proposition \ref{l-regularization}]
Let $g_{i,0}$ be as in the lemma. Given $\delta>0$, there is $i_0$ such that for $i\ge i_0$, then $g_{0,i}$ is $1+ \sigma(n,\delta)$ close to $g_{0,i_0}$ where $ \sigma(n,\delta)$  is the constant obtained from Proposition \ref{hFlow-improve}. We also assume that $\sigma<\e_n$ where $\e_n$ is in the constant in Theorem \ref{Simon-Theorem}. Denote $g_{0,i_0}$ be $h$. Then $h$ is smooth and $\sum_{k=0}^2|\wn^k \Rm(h)|\le k_0$ for some $k_0>0$. By Theorem \ref{Simon-Theorem}, Proposition \ref{hFlow-improve}, and \ref{prop-local-spacetime}, for each $i\ge i_0$ there is a solution $g_i(t)$ to the $h$-flow on $M\times(0,T]$ for some $T>0$ independent of $i$. Moreover
\begin{equation}
\left\{
\begin{array}{ll}
\displaystyle\quad \quad  \quad  \quad \frac12 h\le g_i(t)\le 2h;\\
|\wt \nabla g_i(t)|^2+|\wt\nabla^2 g_i(t)|+|\Rm_{g_i(t)}|\leq \delta t^{-1}
\end{array}
\right.
\end{equation}
%
%
%
%
on $M\times(0,T]$, and if $\Omega\Subset M\setminus \Sigma$, then for $i$ large enough
 we have
$$\sup_{\Omega\times [0,T]}|\wt \nabla^k g_i(t)|\leq C_k$$
for some $C_k>0$ depending only on $n,k, \Omega, g_0$ because $g_{i,0}=g_0$ outside $\Sigma(\frac1i)$. Hence by taking a subsequence, $g_i(t)$ will converge to a solution to the $h$-flow on $M\times(0,T]$ so that
$$|\wt \nabla g (t)|^2+|\wt\nabla^2 g (t)|+|\Rm_{g (t)}|\leq \frac{\delta }{t}$$
and $g(t)$ is smooth up to $t=0$ outside $\Sigma$.  Moreover by the proof of \cite[Theorem 5.2]{Simon2002},
\begin{equation}\label{limiting-hflow-t=0}
\lim_{t\to 0}||g(t)- g_0||_\infty=0.
\end{equation}
This completes the proof.
\end{proof}

\section{A monotonicity formula and a local maximum principle}\label{ss-MP}

We need a local maximum principle from \cite{LeeTam2020}. We only state a weaker form which is sufficient for our purpose.
\begin{prop}\label{t-MP}
Let $h$ be a smooth metric so that
$$
|\Rm(h)|\le k_0,
$$
where $\wn$ is the covariant derivative of the Riemannian connection with respect to $h$.
Suppose $(M,g(t)),t\in [0,S]$ is a smooth solution to the $h$-flow such that $\frac12h\leq g(t)\leq 2 h$ and
$$|\wn g(t)|^2+|\wn^2 g(t)|\leq \frac{\a }{t}$$
on $M\times (0,S]$ for some $\a>1$.

Suppose $\varphi$ is a smooth function on $M\times [0,S]$ such that $\varphi(0)\leq 0$ on $B_{g_0}(x_0,r)$, $\varphi\leq \a t^{-1}$ and
\begin{equation}
\left(\frac{\partial}{\partial t}-\Delta_{g(t)}\right)\varphi \leq \langle W,\nabla \varphi\rangle +L\varphi
\end{equation}
 for some non-negative continuous function $L$ on $M\times [0,S]$ with $L\leq \a t^{-1}$, where $W$ is the vector field as in \eqref{equ:h-flow}. Then for any $l>\a+1$, there exist $S\ge S_1(n,\a,k_0)>0$ and $ T_1(n,\a, k_0,l)>0$ such that for all $t\in [0,S_1\wedge (r^2 T_1)]$,
$$\varphi(x_0,t)\leq 4^{l+1} t^l r^{-2(l+1)}.$$
\end{prop}
\begin{proof}
By the discussion in Section~\ref{Sec: RF}, $\wh g(t)=\Phi^*_t g(t),t\in [0,S]$ is a smooth solution to the Ricci flow with $\wh g(0)=g(0)=g_0$, where $\Phi_t$ is as in \eqref{e-Phi}. Moreover, $\wh\varphi(t)=\varphi\left( \Phi_t(x),t\right),\wh L(x,t)=L\left( \Phi_t(x),t\right)$ satisfy
\begin{equation}
\left(\frac{\partial}{\partial t}-\Delta_{\wh g(t)}\right)\wh\varphi \leq \wh L \wh\varphi.
\end{equation}
with
$$
\wh\varphi(t)\leq \a t^{-1},\quad\text{and}\quad
\wh L(t)\leq \a t^{-1}.
$$
On the other hand, one can check that there is $c_1(n)>0$ such that
$$
|\Rm(g(t))|\le c_1\lf(|\Rm(h)|_h+|\wn g(t)|^2_h+|\wn^2g(t)|_h\ri).
$$
Hence there is $0<S_1<S$ with $S_1=S_1(n, k_0,\a)$ so that
\be
|\Rm(\wh g(t))|=|\Rm(g(t))|\le \frac{2c_1\a}t
\ee
for $t\in (0,S_1]$.
Moreover, we still have $\wh\varphi(0)\leq 0$ on $B_{g_0}(x_0,r)$. By applying \cite[Corollary 3.1]{LeeTam2020} on $B_{g_0}(x,r/2)$ where $x\in B_{g_0}(x_0,r/2)$, we deduce that for any $l>\a+1$, we can find $T_1(n,\a,l)>0$ such that for all  $(x,t)\in B_{g_0}(x_0,r/2)\times [0,S_1\wedge (T_1 r^2)]$,
\begin{equation}
\wh \varphi(x,t)\leq 4^{l+1}t^l r^{-2(l+1)}.
\end{equation}

Moreover by \cite[Corollary 3.3]{SimonTopping2016}, we may shrink $T_1$ further so that
\begin{equation}\label{LMP-RF-lowerbound}
\wh \varphi(x,t)\leq 4^{l+1}t^l r^{-2(l+1)}
\end{equation}
for all $x\in B_{\wh g(t)}(x_0,r/4), t\in [0,S\wedge (T_2 r^2)]$.
\bigskip

Recall that $\partial_t \Phi_t=-W$ with $|W|_h\leq \a t^{-1/2}$,
\begin{equation}
\begin{split}
d_{g(t)}\left(\Phi_t(x_0),x_0 \right)&\leq \a\cdot d_h\left(\Phi_t(x_0),x_0 \right)\\
&\leq  2\a^2\sqrt{t}\\
&\leq \frac{r}{4}.
\end{split}
\end{equation}
provided that $T_2\leq (8\a^2)^{-2}$. Since $\wh g(t)$ is isometric to $g(t)$ through $\Phi_t$, $$x_0\in B_{g(t)}(\Phi_t(x_0),r/4)=\Phi_t\left( B_{\wh g(t)}(x_0,r/4) \right).$$

By \eqref{LMP-RF-lowerbound}, this completes the proof.
\end{proof}

\bigskip

We also need the monotonicity of scalar curvature along the Ricci flow and the Ricci-Deturck flow.
\begin{lma}\label{l-R-monotone-new}
Suppose $(M,g(t)),t\in [0,S]$ is a smooth solution to the Ricci  flow such that
\begin{enumerate}
\item $\sup_{M\times [\tau,T]} |\Rm|<+\infty$ for $\tau\in (0,T]$;
\item $ \cR(g(t))\geq -at^{-1}$ for some $a>0$;
\item there exists $x_0\in M$ and $\Lambda,k>0$ such that $\mathrm{Vol}_{g(t)}(B_{g(t)}(x_0,r)) \leq \Lambda r^k$ for all $r>0$.
\end{enumerate}
Let $\sigma(t)=\sigma_0(1-\frac2n\sigma_0t)^{-1}$ where $\sigma_0\le0$ is a constant, then for all $0< t\leq s\leq S$, if $\varphi\in L^1(M,g(t))$, we have
$$ \left( \int_M \varphi(s) \; d\mu_{ g(s)} \right)\leq \left(\frac{s}{t} \right)^a \left( \int_M \varphi(t) \; d\mu_{ g(t)} \right)$$
where $\varphi(x,t)=\left(\cR_{g(t)}(x) -\sigma(t)\right)_-$.
\end{lma}
\begin{proof}
For any $\theta>0$, let
$$
 v(x,t)=\frac12 \left(\lf(\lf(\cR_{g(t)}-\sigma(t)\ri)^2+\theta\ri)^\frac12-(\cR_{g(t)}-\sigma(t))\right).
$$

We compute
  \bee
\begin{split}
\lf(\frac{\p}{\p t}-\Delta_{g(t)}\ri)v
=&\frac{-v}
{\lf(\cR_{g(t)}-\sigma(t))^2+\theta\ri)^\frac12}\lf(\frac{\p}{\p t}(\cR_{g(t)}-\sigma(t))-\Delta_{g(t)} \cR_{g(t)}\ri)\\
& -\frac{\theta|\nabla \cR_{g(t)}|^2}{2\lf(( \cR_{g(t)}-\sigma(t))^2+\theta\ri)^\frac32}\\
\le &\frac{ v}
{\lf(\cR_{g(t)}-\sigma(t))^2+\theta\ri)^\frac12}\cdot \frac2n\lf(   -\cR_{g(t)}^2+\sigma^2(t)\ri)
\end{split}
\eee

Let $\phi:[0,+\infty)\to\mathbb{R}$ be a non-increasing function such that $\phi=1$ on $[0,1]$, vanishes outside $[0,2]$ and satisfies $|\phi'|\leq 10^4, \phi''\geq -10^4\phi$. On $[\a,\b]\subset (0,S]$, we let $\Phi(x)$ be a cutoff function on $M$ given by $\Phi(x)=\phi^m(\frac{\rho(x)}{R})$ for $R>1$ where $\rho$ is uniformly equivalent to $d_{g(a)}(x,p)$ for some $p\in M$ and $|\partial\rho|^2_{g(\a)}+|\nabla^{2,g(\a)}\rho|\leq C_{\a}$ for some $C_{\a}>1$, obtained from \cite{Tam2010}. If $M$ is compact, we simply take $\phi\equiv 1$.

Since $g(t)$ has bounded curvature on $[\a,\b]$, we have $|\Delta_{g(t)}\rho|\leq C_\a$. Hence,
\begin{equation}
\begin{split}
 &\quad \frac{d}{dt}\left( \int_M v \Phi  \;d\mu_{g(t)} \right)\\
&= \int_M \partial_tv\cdot \Phi - v \Phi\cdot \cR_{g(t)} \;d\mu_{g(t)}\\
&= \int_M  \lf(\frac{\p}{\p t}-\Delta_{g(t)}\ri)v\cdot \Phi +v\Delta_{g(t)}\Phi- v \Phi\cdot \cR_{g(t)} \;d\mu_{g(t)}\\
&\leq \int_M v\Phi\lf(\frac2n\frac{-\cR_{g(t)}^2+\sigma^2(t)}{\lf(\cR_{g(t)}-\sigma(t))^2+\theta\ri)^\frac12}-\cR_{g(t)}\ri) \; d\mu_{ g(t)} +\frac{C_{m,\a}}{R}\int_M v\Phi^{1-\frac1m}\; d\mu_{g(t)}.
\end{split}
\end{equation}

Therefore for $0<\a\leq t<s \leq \b\leq S$,
\begin{equation}
\begin{split}
&s^{-a} \left( \int_M v\Phi \; d\mu_{ g(s)} \right)-t^{-a} \left( \int_M v \Phi\; d\mu_{ g(t)} \right)\\
\le &\int_t^s\tau^a\left[\int_M  v(\tau)\Phi\lf(\frac2n\frac{-\cR_{g(\tau)}^2+\sigma^2(\tau)}{\lf(\cR_{g(\tau)}-\sigma(\tau))^2+\theta\ri)^\frac12}-\cR_{g(\tau)}\ri) \; d\mu_{ g(\tau)}\right]\; d\tau \\
&-\int_t^s a\tau^{-a-1}\left( \int_M v(\tau)\Phi \; d\mu_{ g(\tau)} \right)\;d\mu_{g(\tau)}\; d\tau\\
&+\frac{C_{m,\a}}{R^{1-\frac{k}{m}}} \int^s_t \left(\int_M v \Phi d\mu_\tau \right)^{1-\frac1m}  d\tau
\end{split}
\end{equation}

By letting $\theta\to 0$, $v(\tau)\to \varphi(\tau)$ which is positive only at points where $\cR(\tau)<\sigma(\tau)$ where we have
\bee
 \frac{-\cR_{g(\tau)}^2+\sigma^2(\tau)}{\lf(\cR_{g(\tau)}-\sigma(\tau))^2+\theta\ri)^\frac12}\to\cR_{g(\tau)}+\sigma(\tau)\le \cR(g(\tau)
\eee
because $\sigma(\tau)\le 0$. Since $\cR_{g(\tau)}\geq -a\tau^{-1}$ and $1-2/n<1$, by choosing $m=2k$ we have
\be\label{L1-SC}
\begin{split}
 &s^{-a} \left( \int_M \varphi(s)\Phi \; d\mu_{ g(s)} \right) -t^{-a} \left( \int_M  \varphi(t)\Phi \; d\mu_{ g(t)} \right)
\le \frac{C_{k,\a}}{R^{1/2}} \int^\b_t \left(\int_M v \Phi d\mu_\tau \right)^{1-\frac1{2k}}  d\tau
\end{split}
\ee
for all $0<\a\leq t<s \leq \b\leq S$.

By putting $\a=t$ and integrating $s$ over $[t,\b]$, we see that the integral on the right hand side is finite as $R\to +\infty$ since $\varphi(t)\in L^1(M,g(t))$. Result follows from letting $R\to +\infty$ on \eqref{L1-SC}.
\end{proof}

\section{Singular metrics on compact manifolds}\label{Sec:compact}

  In this section, we will prove Theorem~\ref{main-Thm-Ric} by showing that the scalar curvature lower bound is preserved along the Ricci flow if the co-dimension of the singularity is strictly larger than $2$.   When the initial metric has scalar curvature lower bound in distributional sense and higher regularity, the preservation of scalar curvature lower bound has been studied recently in \cite{JiangShengZhang2021}.

\begin{proof}[Proof of Theorem~\ref{main-Thm-Ric}] Let
\be\label{e-sigma}
\sigma(t)=\sigma_0\lf(1-\frac2n\sigma_0t\ri)^{-1}.
\ee
By Proposition \ref{l-regularization}, one can find a smooth metric $h$ and   $T>0$ so that the $h$-flow \eqref{equ:h-flow} has a solution $g(t)$ in $M\times(0,T]$ satisfying the conditions (i)--(iii) in the Lemma \ref{l-regularization} with $\delta =\frac14 a$.

For any $t_0>0$, apply Lemma \ref{l-R-monotone-new} on the corresponding Ricci flow of $g(t)$ using \eqref{e-Phi}  on $[t_0,T]$ and let $t_0\to 0$, we have the following monotone property:
\be\label{e-monotone-1}
\int_M\varphi(s)\;d\mu_{g(s)}\le \lf(\frac st\ri)^{\frac 14a}\int_M\varphi(t)\;d\mu_{g(t)}
\ee
for all $0<t<s<T$, where $\varphi(x,\tau)=\lf(\cR_{g(\tau)}(x)-\sigma(\tau)\ri)_-$.

We want to prove that $\varphi\equiv0$ on $M\times(0,T]$. By \eqref{e-monotone-1}, it is sufficient to prove that
\be\label{e-R-limit}
\lim_{t\to 0^+}t^{-\frac 14a}\int_M\varphi(t)\;d\mu_{g(t)}=0.
\ee

\bigskip

Fix $\ell>a+1$. Let $t_0>0$ and for any $x_0\in M\setminus\Sigma$ with $d_{g_0}(x_0,\Sigma)=r_0$. We can choose $t_i>0$ with $t_i\to 0$ so that $|d_{g(t_i)}(x,y)-d_{g_0}(x,y)|\le \frac 1i$ for all $x, y\in M$ because $g(t)\to g_0$ in $C^0$ norm as $t\to0$, and
$$
|\cR_{g(t_i)}-\cR_{g_0}|\le \frac1{2i}
$$
in $B_{g_0}(x_0,r_0/2)$ because $g(t)\to g_0$  in $C^\infty_{loc}(M\setminus \Sigma)$.
For any  $0<t_i<T$, by considering the corresponding Ricci flow on $M\times[t_i,T]$ using \eqref{e-Phi} with initial metric $g(t_i)$, we have
\be\label{equ-R-RDF}
\lf(\frac{\p}{\p t}-\Delta_{g(t)}\ri)\cR_{g(t)}\ge \frac2n\cR_{g(t)}^2+\la W,\nabla \cR_{g(t)}\ra
\ee
where $W$ is given by \eqref{equ:h-flow}. Hence

\bee
\lf(\frac{\p}{\p t}-\Delta_{g(t)}\ri)( \sigma(t)-\frac1i-\cR_{g(t)})\le  \la W, \nabla ( \sigma(t)-\frac1i-\cR_{g(t)} )\ra.
\eee
Observe that $ \sigma(t_i)-\frac1i-\cR_{g(t_i)}\le 0$ in $B_{g(t_i)}(x_0,\frac14 r_0)$

Since $g(t)$ satisfies Lemma \ref{l-regularization} (ii), (iii), one can apply Proposition \ref{t-MP} to $g(t)$, for $t\in [t_i,T]$ to conclude that there is $C_1>0$, $T>T_2>T_1>0$ independent of $i$ and $x_0$,  so that if  $t_0\le C_1T_1r_0^2\le T_2$, we have
\bee
\cR_{g(t_0)}(x_0)\ge \sigma(t_0)-\frac1i-C_2t_0^\ell r_0^{-2(l+1)}
\eee
for some constant $C_2$ independent of $t_0, r_0, x_0$ provided $t_0\le C_1T_1r_0^2$. We may choose $T_1$ small enough so that $C_1T_1D^2\le T_2$ where $D$ is the diameter of $M$ with respect to $g_0$.

Let $i\to\infty$, we have
\be\label{e-lowerboundR}
\cR_{g(t_0)}(x_0)\ge \sigma(t_0)-C_2t_0^\ell r_0^{-2(l+1)},
\ee
where $r_0=d_{g_0}(x_0,\Sigma)$ provided that $t_0\leq C_1T_1r_0^2$.
\bigskip

Now for $0<t_0<s$, let $\e_0$ be such that $t_0=C_1T_1\e_0^2\le T_2$.  We want to estimate
   \be\label{split-terms-R}
   \begin{split}
   \int_M \varphi (t_0)d\mu_{g(t_0)}=\lf(\int_{\Sigma(b)}+\int_{M\setminus\Sigma(b)}\ri) \varphi (t_0)d\mu_{g (t_0)}=\mathbf{I}+\mathbf{II}.
   \end{split}
   \ee
   Since we want to estimate this for $t_0$ small, we may assume that $\e_0\le \frac b2$.
    By \eqref{e-lowerboundR}, 
   there is $C_2>0$ such that
   $$
   \varphi (x,t)\le C_2t
   $$
   for some constant $C_2$ for all  $(x,t)\in (M\setminus \Sigma(b))\times[0,T]$. 
   Hence
   \be\label{e-II}
   \mathbf{II}\le  C_2t_0 V_{g(t_0)}(M)\le C_3 \e_0^2
   \ee
   for some constant $C_3$ independent of $t_0$. 
%
%
%

To estimate $\mathbf{I}$, let $k$ be positive integer so that
$$
2^k\e_0\le b\le 2^{k+1}\e_0.
$$
Then
\bee
\begin{split}
\mathbf{I}=\sum_{j=1}^{k+1} \mathbf{I}_j
\end{split}
\eee
where
$$
\mathbf{I}_j=\int_{\Sigma(2^j\e_0)\setminus \Sigma(2^{j-1}\e_0)}\varphi(t_0)d\mu_{g(t_0)},
$$
for $j\ge 2$ and

$$
\mathbf{I}_1=\int_{\Sigma(2\e_0)}\varphi(t_0)d\mu_{g(t_0)}.
$$
By the assumption on $\Sigma$,
$$
\mathbf{I}_1\le at_0^{-1} (2\e_0)^{2+a}\le C_4\e_0^a
$$
for some constant $C_4$ independent of $\e_0$.
For $k\ge j\ge2$,   we apply \eqref{e-lowerboundR} to obtain 
\bee
\begin{split}
 \mathbf{I}_j\le&  C_2t_0^\ell (2^{j-1}\e_0)^{-2(l+1)} (2^j\e_0)^{(2+a)}\\
 \le & C_5  2^{(-2l+a)j}\e_0^a\\
 \le &C_5 2^{-lj}\e_0^a
 \end{split}
\eee
for some $C_5$ independent of $\e_0, j$ because $\ell>a+1$.
 Also, as before, we have

\bee
\mathbf{I}_{k+1}\le C_6\e_0^2.
\eee
for some constant $C_0$ independent of $\e_0$. Therefore,
\be\label{esti-I}
\begin{split}
\mathbf{I}\le & C_4\e_0^a+C_6\e_0^2+C_5\e_0^a\sum_{j=2}^k2^{-jl}
\le C_7\e_0^a
\end{split}
\ee
for some $C_7$ independent of $t_0$, provided $0<a<1$ and $\e_0\le 1$.
Assuming $0<a<1$, $\e_0<1$. Combining with \eqref{e-II}, we have
\bee
t_0^{-\frac14a}\int_M \varphi(t_0)d\mu_{g(t_0)}\le C_8t_0^{\frac14a}.
\eee
for some $C_8$ independent of $\e_0$   and hence $t_0$.  From this one can conclude  \eqref{e-R-limit} and hence $\cR_{g(t)}\ge  \sigma(t)$ for all $t\in (0,T]$ by \eqref{e-monotone-1}.

By considering the corresponding Ricci flow for $t>0$, we have
\bee
\begin{split}
\frac{d}{dt}\mathrm{Vol}(M,g(t))&=\int_M -\cR_{g(t)}\;d\mu_{g(t)}\\
&\le -\sigma(t)\mathrm{Vol}(M,g(t))\\
&=-\sigma_0(1-\frac2n\sigma_0t)^{-1}\mathrm{Vol}(M,g(t)).
\end{split}
\eee
for all $t>0$. Since $\mathrm{Vol}(M,g(t))\to \mathrm{Vol}(M,g_0)=1$ as $t\to 0$, we have
$$
\mathrm{Vol}(M,g(t))\le \lf(1-\frac2n\sigma_0t\ri)^\frac n2.
$$

\bigskip

For each $t>0$, consider normalized metric $\wt g(t)=(\mathrm{Vol}(M,g(t)))^{-\frac 2n}g(t)$ so that  $\mathrm{Vol}(M,\tilde g(t))=1$, and
\bee
\begin{split}
\cR_{\wt g(t)}=(\mathrm{Vol}(M,g(t)))^{\frac 2n}\cR_{g(t)}&\ge (\mathrm{Vol}(M,g(t)))^{\frac 2n}\sigma(t)\\
&=
(\mathrm{Vol}(M,g(t)))^{\frac 2n}\lf(1-\frac2n\sigma_0t\ri)^{-1}\sigma_0\ge  \sigma_0
\end{split}
\eee
because $(\mathrm{Vol}(M,g(t)))^{\frac 2n}\lf(1-\frac2n\sigma_0t\ri)^{-1}\le 1$ and $\sigma_0\le0$.

\bigskip
 It is well-known that on a compact manifold $M$ with $\sigma_0=\sigma(M)\leq 0$, any smooth metric $g$ with unit volume and $\cR_g\geq \sigma_0$ must be Einstein with
$\cR_g= \sigma_0$, see \cite[p.126-127]{Schoen1987}. Therefore, $\wt g(t)$ is Einstein with $\cR_{\wt g(t)}=\sigma_0$. By rescaling it back, we have
\begin{equation}\label{estimate-h-flow-limiting}
\left\{
\begin{array}{ll}
\mathrm{Ric}_{g(t)} =\sigma_0 \left(n-2\sigma_0 t\right)^{-1} g(t);\\
\mathrm{Vol}\left(M,g(t)\right) = \left(1-2n^{-1}\sigma_0 t\right)^{n/2}
\end{array}
\right.
\end{equation}
for $t\in (0,T]$.  By \eqref{equ:h-flow}, the normalized flow $\wh g(t)=\displaystyle \frac{n}{n-2\sigma_0 t}g(t),t\in (0,T]$ satisfies
\begin{equation}
\left\{
\begin{array}{ll}
\partial_t \wh g_{ij}= \wh\nabla_i \wh W_j + \wh\nabla_j \wh W_i\\
\wh W_j=\displaystyle  \frac{n}{n-2\sigma_0t} \cdot \wh g_{jk} \wh g^{pq} \left(\wh\Gamma_{pq}^k-\wt \Gamma_{pq}^k \right)
\end{array}
\right.
\end{equation}
where $\wt \Gamma,\wh \Gamma$ are the Christoffel symbols of $h$ and $\wh g(t)$ respectively. By considering the ODE:
\begin{equation}
\left\{
\begin{array}{ll}
\partial_t\Psi_t(x)= \wh W\left( \Psi_t(x),t\right);\\
\Psi_T(x)=x
\end{array}
\right.
\end{equation}
for $(x,t)\in M\times (0,T]$. We obtain a family of diffeomorphisms $\Psi_t$. Moreover, one can check that
 $$
 \frac{\p}{\p t}(\Psi_t^{-1})^* \wh g(t)=0.
 $$
 Hence $(\Psi_t^{-1})^* \wh g(t)=(\Psi_T^{-1})^* \wh g(T)=\wh g(T)$
and  $\wh g(t)=\Psi_t^* \wh g(T)$ for $t\in (0,T]$. On the other hand, $|\wh W|_h\leq Ct^{-1/2}$, we conclude that $\Psi_t\to \Psi_0$ as $t\to 0$ in $C^0$ norm for some continuous map $\Psi_0: M\to M$. Since $\wh g(t)\to g_0$ in $C^0$ as $t\to 0$, we have
\begin{equation}
\begin{split}
d_{g_0}(x,y)&=\lim_{t\to 0} d_{\wh g(t)}(x,y)\\
&=\lim_{t\to 0}
d_{\wh g(T)}\left(\Psi_t(x),\Psi_t(y) \right)=d_{\wh g(T)}\left(\Psi_0(x),\Psi_0(y) \right)
\end{split}
\end{equation}
for all $x,y\in M$. This in particular shows that $\Psi_0$ is a   homeomorphism of $M$.
Since $g_0$ is uniformly equivalent to $\wh g(T)$, we have
$$
C^{-1}d_{\wh g(T)}(x,y)\le d_{\wh g(T)}\left(\Psi_0(x),\Psi_0(y) \right)\le Cd_{\wh g(T)}(x,y)
$$
for some positive constant $C$ and for all $x, y\in M$.
Following \cite[(5.2)]{CalabiHartman1970}, $\Psi_t$ also satisfies
\begin{equation}
\frac{\partial^2 \Psi_t^m}{\partial x^i \partial x^j}=\wh \Gamma_{ij}^k \frac{\partial \Psi_t^m}{\partial x^k}-\bar \Gamma_{kl}^m \frac{\partial \Psi_t^l}{\partial x^i} \frac{\partial \Psi_t^k}{\partial x^j}
\end{equation}
in local coordinate of $M$ where $\Psi_t^m$ are the components of $\Psi_t$, $\bar \Gamma$ is the Christoffel symbol of $\wh g(T)$. Since $\wh g(t)$ is smooth up to $t=0$ outside $\Sigma$, $\Psi_t$ is bounded in $C^\infty_{loc}(M\setminus \Sigma)$ as $t\to 0$. And hence, $\Psi_0$ is smooth and satisfies $\Psi_0^*\wh g(T)=g_0$ outside $\Sigma$. Since $\wh g(T)$ is Einstein with scalar curvature $\sigma_0$ and with unit volume, this completes the proof of (i), (ii) of the theorem.

Suppose $\Sigma$ consists only of isolated singular points. Since $\wh g(t)=\Psi^*_t(\wh g(T))$, the curvature of $\wh g(t)$ are uniformly bounded independent of space and time. Therefore, the curvature of $g_0$ is also uniformly bounded on $M\setminus \Sigma$.  Recall that $g_0$ is Einstein outside $\Sigma$. The last assertion of the theorem follows from the removable singularity result of Smith-Yang \cite{SmithYang1992}.
\end{proof}
\begin{cor}\label{c-points} With the same assumptions on $(M^3,g)$, $\Sigma$ as in Theorem \ref{main-Thm-Ric}  where $\Sigma$ consists of countable many points $\{p_k\}$ with one limit point $p$. Suppose the co-dimension of $\Sigma$ is larger than 2, then $g$ can be extended to be a smooth metric which is Einstein.

\end{cor}
\begin{proof} By Theorem \ref{main-Thm-Ric}, $g$ is Einstein outside $\Sigma$. By \cite{SmithYang1992}, $g$ can be extended to be smooth near each $p_k$ after possible change of smooth structure. However, in dimension three, the smooth structure is unique. Hence $g$ can be extended smoothly near each $p_k$, and $g$ has only one possible singularity $p$. This is also removable using \cite{SmithYang1992} again.

\end{proof}
\begin{rem}
By (i) of the main Theorem in \cite{CalabiHartman1970}, it is reasonable to expect $\Psi_0$ not to be in $C^1$  in general unless we have stronger degree of continuity on $g_0$.  For example, if $g_0$ satisfies certain Dini continuity condition, then $\Psi_0$ is $C^1$ and in this case we have $g_0=\Psi_0^*\wh g(T)$ on the whole manifold $M$. 
\end{rem}
%

As a corollary of the proof of Theorem~\ref{main-Thm-Ric}, we can show that if $g_0$ is a continuous metric which is smooth away from singularity of high co-dimension, then $g_0$ is of scalar curvature lower bound in the weak sense introduced by Burkhardt-Guim in \cite{Burkhardt2019}.
\begin{cor}\label{Cor:scalar-lower-C0}
Let $M^n$ be a compact manifold and $g_0$ is a continuous metric on $M$ such that $g_0\in C^\infty_{loc}(M\setminus \Sigma)$ for some compact set $\Sigma$ of co-dimension at least $2+a$ for some $a>0$ and $\mathcal{R}(g_0)\geq \sigma_0$ for some $\sigma_0\in\mathbb{R}$ on $M\setminus \Sigma$, then there is a family of smooth metric $g(t),t\in (0,T]$ with $\mathcal{R}(g(t))\geq \sigma_0$ on $M$ such that $g(t)\to g_0$ in $C^0(M)$ as $t\to 0$. In particular, $g_0$ has scalar curvature bounded from below by $\sigma_0$ in the $\b$-weak sense for any $\b<1/2$.
\end{cor}
\begin{proof}
 By the proof of Theorem~\ref{main-Thm-Ric}, there is a smooth solution  $g(t)$ to the $h$-flow on $M\times (0,T]$ for some $h$ such that $g(t)\to g_0$ in $C^0(M)$ as $t\to 0$ and $$\cR(g(t))\geq \sigma_0$$
for $t\in (0,T]$ where we have used $\heat \cR\geq 0$. Result follows by rescaling and  \cite[Corollary 1.6]{Burkhardt2019}. We note here that since we only need to construct sequence of $g_i\to g_0$ with scalar curvature lower bound, we don't need the sharpest estimate on the lower bound.
\end{proof}

\begin{rem}
When $\sigma_0=0$, the scalar curvature doesn't play a crucial role in the analysis as we are not required to control the volume in this case. Similar results will hold as long as the uniform local estimate \eqref{e-lowerboundR} is true. In particular, this is the case if $\Rm(g_0)(x)$ is inside the curvature cone studied in Ricci flow theory for all $x\notin \Sigma$, see \cite[Theorem 3.1]{LeeTam2020}. For example, if initially $g_0$ is of weakly $\mathbf{PIC}_1$ away from some compact sets of co-dim $\geq 2+a$ for some $a>0$, then the Ricci-Deturck flow will give a smooth metric with weakly $\mathbf{PIC}_1$ globally on $M$.
\end{rem}

\section{Positive mass theorem with singular set}\label{Sec:AF}

In this subsection, we will discuss the analogy of subsection~\ref{subsection-compact-regularize} in the asymptotically flat setting. There are different definitions for asymptotically flat manifold. For our purpose, we consider the following:
\begin{defn}\label{AF-condition}
An n dimensional Riemannian manifold $(M^n, g)$, where $g$ is continuous, is said to be asymptotically flat (AF) if there is a compact subset $K$ such that $g$ is smooth on $M\setminus K$, and $M\setminus K$ has finitely many components $E_k$, $1\leq k\leq l$, each $E_k$ is called an end of $M$, such that each $E_k$ is diffeomorphic to $\mathbb{R}^m\setminus B_{euc}(R)$ for some Euclidean ball $B_{euc}(R)$, and the following are true:
\begin{enumerate}
\item[(i)] In the standard coordinates $x^i$ of $\mathbb{R}^n$, $g_{ij}=\delta_{ij}+\sigma_{ij}$
so that
$$\sup_{E_k}\left\{\sum_{s=0}^2 |x|^{\tau+s}|\partial^s \sigma_{ij}| + \left[|x|^{\a+2+\tau} \partial^2 \sigma_{ij} \right]_\a \right\}<+\infty$$
for some $0<\a\leq 1$, $\tau>\frac{n-2}{2}$, where $\partial^s f$ denotes the derivatives
 with respect to the Euclidean metric, and $[f]_\a$ is the $\a$-H\"older norm of $f$;
\item[(ii)]The scalar curvature $\cR$ satisfies the decay condition:
$$|\cR(x)|\leq C(1+d_g(x,p))^{-q}$$
for some $C>0$ and $n+2\geq q>n$. Here $d_g(x,p)$ is the distance function from $p\in M$.
\end{enumerate}
\end{defn}

The coordinate chart satisfying (i) is said to be admissible. In the following, for a function $f$ defined near infinity of $\mathbb{R}^n$. For $k\geq 0$, we say that $f=O_k(r^{-\tau})$ if $\sum_{i=0}^k r^i |\partial^if|=O(r^{-\tau})$ as $r=|x|\to +\infty$.

\begin{defn}[\cite{ArnowittDeserMisner1961}]
The Arnowitt-Deser-Misner (ADM) mass of an end $E$ of an AF manifold $M$ is defined as
 $$m_{ADM}(E)=\lim_{r\to +\infty}\frac{1}{2(n-1)\omega_{n-1}}\int_{S_r}(g_{ij,i}-g_{ii,j})\nu^j \; dA^0_r$$
 in an admissible coordinate chart where $S_r$  is the Euclidean sphere, $\omega_{n-1}$ is the volume of $n-1$ dimensional unit sphere, $dA^0_r$  is the volume element induced by the Euclidean metric, $\nu$ is the outward unit normal of $S_r$ in $\mathbb{R}^n$ and the derivative is the ordinary partial derivative.
\end{defn}

By the result of Bartnik \cite{Bartnik1986},  $m_{ADM}(E)$ is independent of the choice of admissible charts and hence it is well-defined. we have the following positive mass theorem by Schoen and Yau \cite[Theorem 5.3]{SchoenYau2017}, see also \cite{Bartnik1986,ParkerTaubes1982,Schoen1987,SchoenYau1979,SchoenYau1979-2,Witten1981}.
\begin{thm}\label{SY-PMT}
Assume that $(M,g)$ is an AF manifold with $\cR(g)\geq 0$. For each end $E$, we have $m_{ADM}(E)\geq 0$. Furthermore, if $m_{ADM}(E)=0$ for some end $E$, then $(M,g)$ is isometric to $\mathbb{R}^n$.
\end{thm}

 We want to prove the positive mass theorem for metrics which are
smooth outside a compact set of codimension at least $2+a$ for some $a>0$.  We begin with the regularization using $h$-flow.
\begin{prop}\label{AF-hflow}
Under the assumption of Proposition~\ref{l-regularization}, if in addition $g_0$ is AF, then there is a smooth metric $h$ on $M$ satisfying \eqref{background-h-regular} such that the $h$-flow has a solution $g(t)$ on $M\times (0,T]$ which satisfies the conclusions in Proposition~\ref{l-regularization} and is AF. Moreover, for any end $E$, the mass of $g(t)$ satisfies
\begin{equation}\label{ADM-flow-ineq}
m_{ADM,g(t)}(E)\leq m_{ADM,g_0}(E).
\end{equation}
\end{prop}
\begin{proof}
As in the proof of Proposition~\ref{l-regularization}, we let $\sigma(n,\delta)$ be the constant obtained from Proposition~\ref{hFlow-improve}. By Lemma~\ref{smooth-approximation},  there is a sequence of smooth metrics $g_{i,0}$ such that
\begin{enumerate}
\item[(i)] $g_{i,0}\to g_0$ in $C^0$ uniformly;
\item[(ii)] $g_{i,0}=g_0$ outside $\Sigma(i^{-1})$.
\end{enumerate}
In particular,  since $g_{i,0}$ coincides with $g_0$ outside a compact set, $g_{i,0}$ is uniformly AF for all $i\in \mathbb{N}$. Namely, $g_{i,0}$ satisfies the same estimates in Definition~\ref{AF-condition}.

Since $g_{i,0}$ is uniformly AF,  we may choose $h=\phi g_{i_0,0}+(1-\phi) g_{euc}$ where $\phi$ is a cutoff function on large compact set so that for all $i>i_0$, $g_{i,0}$ is $1+\sigma$ close to $h$ and hence there is a solution $g_i(t)$ to the $h$-flow on $M\times (0,T]$ for some $T$ independent of $i$. We may assume $T<1$. The estimates follows from Theorem \ref{Simon-Theorem}, Proposition \ref{hFlow-improve}, and \ref{prop-local-spacetime}.

To show that $g(t)$ is AF,  it suffices to show that $g_i(t)$ satisfies the condition in Definition~\ref{AF-condition} uniformly in $i>i_0$.  Since $g_{i,0}=g_0$  and $h=g_{euc}$ outside a large compact set,   the proof of \cite[Lemma 7.6]{ShiTam2016} can be carried over, see also \cite{McFeronSzekelyhidi2012,DaiMa2007}. By letting $i\to +\infty$, we obtain AF of $g(t)=\lim_{i\to +\infty} g_i(t)$ for $t>0$.

It remains to establish the inequality relating the mass of $g(t)$ and $g_0$.  Recall that $g_{i,0}$ coincides with $g_0$ outside compact set and the mass is preserved under the smooth Ricci flow (and hence the $h$-flow) by \cite[Corollary 12]{McFeronSzekelyhidi2012},  we have for all $t\in (0,T]$ and $i>i_0$,
\begin{equation}\label{ADM-inequality-1}
\begin{split}
m_{ADM,g_0}(E)&=\liminf_{i\to +\infty} m_{ADM,g_{i,0}}(E)\\
&=\liminf_{i\to +\infty}  m_{ADM,g_{i}(t)}(E).
\end{split}
\end{equation}

Using \cite[(18)]{McFeronSzekelyhidi2012}, AF of $g(t)$ and the fact that $g_i(t)\to g(t)$ in $C^\infty_{loc}(M\times (0,T))$ as $i\to+\infty$,  we have
\begin{equation}\label{ADM-inequality}
\begin{split}
m_{ADM,g(t)}(E)&\leq \liminf_{i\to +\infty}m_{ADM,g_i(t)}(E)+ O(r^{-\lambda}) \\
&\quad +\limsup_{i\to +\infty} \int_{E\setminus B_r} \cR_-(g_i(t)) \; d\mu_{g_i(t)}
\end{split}
\end{equation}
for some $\lambda>0$ where $B_r$ denotes the Euclidean ball of radius $r$ on the end $E$.  Since $g_i(t)$ is uniformly equivalent to  the Euclidean metric outside compact set,  as in \eqref{e-lowerboundR},  for any $l>1+\delta$ there is $C_0>0$ such that for $r\to +\infty$ and $i>i_0$,
\begin{equation}\label{R-decay-flow}
 \cR_-(g_i(t)) \leq C_0 t^l r^{-2(l+1)}
\end{equation}
on $\partial B_r\times (0,T]$. Here we have used the fact that $g_{i,0}=g_0$ outside compact set.  By choosing $l>n$ and using \eqref{R-decay-flow}, we conclude that for all $i>i_0$,
\begin{equation}\label{L1-R}
\begin{split}
 \int_{E\setminus B_r} \cR_-(g_i(t)) \; d\mu_{g_i(t)}&\leq C_1 \int_{E\setminus B_r} \cR_-(g_i(t)) \; d\mu_{g_{euc}}\\
 &\leq C_1t^n\int^\infty_r  s^{-2(l+1)+n-1}  ds\\
 &\leq C_2t^n  r^{-n-2}
\end{split}
\end{equation}
for some $C_2>0$ independent of $i\to+\infty$.\bigskip

Combines this with \eqref{ADM-inequality}, we conclude that
$$m_{ADM,g(t)}(E)\leq \liminf_{i\to +\infty}m_{ADM,g_i(t)}(E)+ O(r^{-\lambda'})$$
for some $\lambda'>0$. This completes the proof by \eqref{ADM-inequality-1} and letting $r\to+\infty$.

\end{proof}

Now we are ready to prove the positive mass theorem with singular metrics.
\begin{proof}[Proof of Theorem~\ref{main-Thm-PMT}]
By Proposition~\ref{AF-hflow}, it suffices to show that $\cR(g(t))\geq 0$. Suppose $\cR(g(t))\geq 0$, Theorem~\ref{SY-PMT} implies $m_{ADM,g(t)}(E)\geq 0$ for $t\in (0,T]$ and hence $m_{ADM,g_0}(E)\geq 0$ by \eqref{ADM-flow-ineq}. Moreover if $m_{ADM,g_0}(E)=0$ for some end $E$, we have $m_{ADM,g(t)}(E)= 0$ and hence $(M,g(t))$ is isometric to the Euclidean space. Moreover, since $g(t)\to g_0$ in $C^\infty_{loc}(M\setminus \Sigma)$ as $t\to0$, $g_0$ is flat outside $\Sigma$. The isometry follows from the fact that $g(t)\to g_0$ in $C^0_{loc}$ as $t\to 0$, see also \cite{JiangShengZhang2020} for the method using RCD theory.

To show that $\cR(g(t))\geq 0$,  we will modify the proof of Theorem~\ref{main-Thm-Ric}. Since $g(t)$ is uniformly equivalent to $h$ and $h=g_{euc}$ outside compact set, letting $i\to+\infty$ in \eqref{L1-R} implies that $\cR_-(g(t))\in L^1(M)$ for $t\in (0,T]$. By Lemma~\ref{l-R-monotone-new}, it suffices to show that
\be\label{e-R-limit-noncompact}
\lim_{t\to 0^+}t^{-\frac 14a}\int_M\cR_-(g(t))\;d\mu_{g(t)}=0.
\ee
\bigskip

We split it into three parts as in \eqref{split-terms-R}:
\begin{equation}
\begin{split}
 &\quad   \int_M \cR_-(g(t_0))\; d\mu_{g(t_0)}\\
 &=\lf(\int_{\Sigma(b)}+\int_{B_{h}(x_0,2R)\setminus\Sigma(b)}+\int_{M\setminus B_{h}(x_0,2R)}\ri) \cR_-(g(t_0))\;d\mu_{g (t_0)}\\
   &=\mathbf{I}+\mathbf{II}+\mathbf{III}
   \end{split}
\end{equation}

Since $h=g_{euc}$ on each end $E_k$, we may choose $R$ sufficiently large such that $$M\setminus B_{h}(x_0,2R)\subset \sqcup_{k=1}^N \lf(E_k\setminus B_{R}\ri)$$
where $B_R$ is the Euclidean ball of radius $R$ and $\{E_k\}_{k=1}^N$ are the ends of $M$.

Using Fatou's lemma and the fact that $g_i(t)\to g(t)$ in $C^\infty_{loc}$, we may pass $i\to\infty$ in \eqref{L1-R} to conclude that
\begin{equation}
\mathbf{III} \leq C_0 t_0^n.
\end{equation}

The estimates of $\mathbf{I}$ and $\mathbf{II}$ follows from the same argument of \eqref{esti-I} and \eqref{e-II} respectively. Namely, we have
\begin{equation}
\mathbf{I}+\mathbf{II}\leq C_1 t_0^{a/2}.
\end{equation}

To conclude, we show that
\begin{equation}
  \int_M \cR_-(g(t_0))\; d\mu_{g(t_0)}=\mathbf{I}+\mathbf{II}+\mathbf{III}\leq C_2 t_0^{a/2}.
\end{equation}
for some $C_2>0$ independent of $t_0\in (0,T]$. This proves \eqref{e-R-limit-noncompact} and hence completes the proof.
\end{proof}

\end{document}